\documentclass[11pt,leqno]{amsart}
\usepackage{a4wide}
\usepackage{amssymb,upgreek}
\usepackage{amsmath}
\usepackage{amsthm}
\usepackage{mathrsfs}
\usepackage{euscript,csquotes}
\usepackage{xypic}
\usepackage[usenames,dvipsnames]{color}
\usepackage{endnotes,comment}

\usepackage{hyperref}

\theoremstyle{plain}

\newcommand{\id}{\operatorname{id}}

\newcommand{\Hom}{\operatorname{Hom}}
\newcommand{\End}{\operatorname{End}}

\newcommand{\Mod}{\operatorname{Mod}}
\newcommand{\Ind}{\operatorname{Ind}}
\newcommand{\ind}{\operatorname{ind}}

\newtheorem{theorem}{Theorem}[section]
\newtheorem{corollary}[theorem]{Corollary}
\newtheorem{lemma}[theorem]{Lemma}
\newtheorem{proposition}[theorem]{Proposition}

\theoremstyle{definition}

\title{\textbf{Derived parabolic induction}}
\author{Sarah Scherotzke, Peter Schneider}
\date{\today}

\address{Department of Mathematics\\
Universit\'{e} du Luxembourg \\
Maison du Nombre\\
6, Avenue de la Fonte\\
L-4364 Esch-sur-Alzette, Luxembourg}
\email{sarah.scherotzke@uni.lu}

\address{Universit\"{a}t M\"{u}nster,  Mathematisches Institut,  Einsteinstr. 62,
48149 M\"{u}nster,  Germany,
 http://www.uni-muenster.de/math/u/schneider/ }
\email{pschnei@wwu.de }

\begin{document}

\begin{abstract}
The classical parabolic induction functor is a fundamental tool on the representation theoretic side of the Langlands program. In this article, we study its derived version. It was shown by the second author that the derived category of smooth $G$-representations over $k$,  $G$ a $p$-adic reductive group and $k$ a field of characteristic $p$, is equivalent to the derived category of a certain differential graded $k$-algebra $H_G^\bullet$, whose zeroth cohomology is a classical Hecke algebra. This equivalence predicts the existence of a derived parabolic induction functor on the dg Hecke algebra side, which we construct in this paper. This relies on the theory of six-functor formalisms for differential graded categories developed by O.\ Schn\"urer. We also discuss the adjoint functors of derived parabolic induction.
\end{abstract}

\maketitle

\section{Introduction}

The smooth representation theory of $p$-adic reductive groups $G$ is one of the cornerstones of the local Langlands program. A very important technique to construct such representations is the parabolic induction. For this one chooses a parabolic subgroup $P \subseteq G$ with Levi decomposition $P = MN$. Starting from a smooth representation $V$ of the Levi factor $M$ one first inflates $V$ to $P$ and then induces further from $P$ to $G$ (using locally constant functions $G \rightarrow V$) in order to obtain a smooth representation $\Ind_P^G(V)$ of $G$. This construction is easily seen to be an exact functor. It therefore passes for trivial reasons to a functor between derived categories. This seems to make the choice of title for this paper a bit strange.

But there is another basic tool to investigate smooth representations which is the notion of Hecke algebras. Here one starts with a choice of some well understood ``big'' smooth $G$-representation $\mathbf{X}$ and forms the (opposite of the) algebra $H_G$ of its $G$-endomorphisms, so that $\mathbf{X}$ becomes a $(G,H_G)$-bimodule. One then has a pair of adjoint functors $\Hom_G(\mathbf{X},-)$ and $\mathbf{X} \otimes_{H_G} -$ between the category of smooth $G$-representations and the category of (left) $H_G$-modules. These may be used to reduce representation theoretic questions to purely algebraic questions about modules. It turns out, though, that the behaviour of these two functor depends very much on the characteristic of the coefficient field $k$ which one uses for the $G$-representations. If $k$ has characteristic zero then the connection between smooth $G$-representations and Hecke modules is very close and has been studied intensively.

With the emergence of a $p$-adic local Langlands program the case of a coefficient field of characteristic $p$ has become increasingly important. But in this case the connection  between representations and Hecke modules is rather weak. To rescue the situation one has to pass to the derived endomorphism algebra $H_G^\bullet$ of $\mathbf{X}$. For appropriate choices of $\mathbf{X}$ derived versions of the above functors even lead to an equivalence of triangulated categories between the derived category $D(G)$ of smooth $G$-representations and the derived category $D(H_G^\bullet)$ of the differential graded $k$-algebra $H_G^\bullet$ (cf.\ \cite{Sch}). Therefore the easily constructed parabolic induction $\Ind_P^G$ on the $D(G)$-side has to correspond to a specific functor on the $D(H_G^\bullet)$-side. This paper is about determining concretely this latter functor on the $H_G^\bullet$-side.

In the first section we will begin by reviewing the underived formalism of parabolic induction. We then will reformulate it on the Hecke side in such a way that it becomes apparent how to write down a derived candidate for the functor. In the second section we will show that our candidate on the $H_G^\bullet$-side indeed corresponds to the functor $\Ind_P^G$ on the derived representation theoretic side. Technically this will require the enhancement of the triangulated category $D(G)$ to a differential graded category. In the final section we discuss what is known about left and right adjoint functors of derived parabolic induction. They do exist, but their explicit computation remains an open problem.

\section{The setting}\label{sec:setting}

Let $k$ be any field of positive characteristic $p$. We fix a locally compact nonarchimedean field $\mathfrak{F}$ of residue characteristic $p$. Let $\mathbf{G}$ be a connected reductive group over $\mathfrak{F}$, and put $G := \mathbf{G}(\mathfrak{F})$. By $\Mod_k(G)$ we denote the Grothendieck abelian category of smooth representations of $G$ in $k$-vector spaces, and by $D(G)$ the corresponding unbounded derived category. For any $\mathfrak{F}$-parabolic subgroup $\mathbf{P} \subseteq \mathbf{G}$ with Levi factor $\mathbf{M}$ and $P := \mathbf{P}(\mathfrak{F})$ and $M := \mathbf{M}(\mathfrak{F})$ we have the usual parabolic induction functor
\begin{equation*}
  \Ind_P^G : \Mod_k(M) \longrightarrow \Mod_k(G) \ .
\end{equation*}
We recall its properties (\cite{Vig} Prop.\ 4.2 and Thm.\ 5.3):
\begin{itemize}
  \item[(A)] $\Ind_P^G$ is exact.
  \item[(B)] The functor of $N$-coinvariants $(-)_N : \Mod_k(G) \longrightarrow \Mod_k(M)$, where $N := \mathbf{N}(\mathfrak{F})$ for the unipotent radical $\mathbf{N}$ of $\mathbf{P}$, is left adjoint to $\Ind_P^G$.
  \item[(C)] $\Ind_P^G$ commutes with arbitrary direct sums, and hence has a right adjoint functor $R_P^G : \Mod_k(G) \longrightarrow \Mod_k(M)$.
  \item[(D)] $\Ind_P^G$ is fully faithful.
  \item[(E)] $\id \xrightarrow{\simeq} R_P^G \circ \Ind_P^G$ and $(-)_N \circ \Ind_P^G \xrightarrow{\simeq} \id$.
\end{itemize}
The property $(A)$ immediately implies that $\Ind_P^G$ extends to an exact functor
\begin{equation*}
  \Ind_P^G : D(M) \longrightarrow D(G)
\end{equation*}
between the unbounded derived categories.

Next we turn to the Hecke algebra side. We fix a maximal $\mathfrak{F}$-split torus $\mathbf{T} \subseteq \mathbf{P}$ as well as a minimal $\mathfrak{F}$-parabolic subgroup $\mathbf{T} \subseteq \mathbf{B} \subseteq \mathbf{P}$. We may view $\mathbf{M}$ in a unique way as a subgroup of $\mathbf{P}$ which contains $\mathbf{T}$. We also fix a special vertex in the apartment corresponding to $\mathbf{T}$ in the semisimple Bruhat-Tits building of $\mathbf{G}$. It is a vertex of a unique chamber in this apartment in the ``direction'' of $\mathbf{B}$. We denote by $I_G$ the pro-$p$-Sylow subgroup of the Iwahori subgroup corresponding to this chamber (called the pro-$p$ Iwahori subgroup of $G$). Then $I_M := M \cap I_G$ is a pro-$p$ Iwahori subgroup of $M$. We introduce the ``universal'' representations
\begin{equation*}
  \mathbf{X}_G := \ind_{I_G}^G(1)  \qquad\text{and}\qquad   \mathbf{X}_M := \ind_{I_M}^M(1)
\end{equation*}
in $\Mod_k(G)$ and $\Mod_k(M)$, respectively. The corresponding pro-$p$ Iwahori-Hecke algebras are
\begin{equation*}
  H_G := \End_{k[G]}(\mathbf{X}_G)^{\mathrm{op}}   \qquad\text{and}\qquad    H_M := \End_{k[M]}(\mathbf{X}_M)^{\mathrm{op}} \ .
\end{equation*}
In \cite{OV} \S2.5.2 a subalgebra $H_{M+} \subseteq H_M$ is defined (which depends on $P$) and an explicit embedding $H_{M+} \hookrightarrow H_G$ is constructed. Then they define (loc.\ cit. \S4.2.1) the parabolic induction functor in this setting by
\begin{align*}
  \Ind_{H_M}^{H_G} : \Mod(H_M) & \longrightarrow \Mod(H_G) \\
                            Y & \longmapsto H_G \otimes_{H_{M+}} Y \ .
\end{align*}
We need a more conceptual description of this functor. Note that obviously $H_G \otimes_{H_{M+}} Y = (H_G \otimes_{H_{M+}} H_M) \otimes_{H_M} Y$. Hence we have to understand the $(H_G,H_M)$-bimodule $H_G \otimes_{H_{M+}} H_M$. For this we use \cite{OV} Prop.\ 4.3 which says that the diagram
\begin{equation}\label{diag:Ind-invariants}
  \xymatrix{
    \Mod_k(M) \ar[d]_{\Ind_P^G}  \ar[rr]^{(-)^{I_M}} && \Mod(H_M) \ar[d]^{\Ind_{H_M}^{H_G}} \\
    \Mod_k(G) \ar[rr]^{(-)^{I_G}} &&   \Mod(H_G)  }
\end{equation}
is commutative (up to natural isomorphism). If we apply this to $\mathbf{X}_M$ in $\Mod_k(M)$ then we obtain isomorphisms of $(H_G,H_M)$-bimodules
\begin{equation*}
  H_G \otimes_{H_{M+}} H_M = \Ind_{H_M}^{H_G}((\mathbf{X}_M)^{I_M}) \cong \Ind_P^G(\mathbf{X}_M)^{I_G} \ .
\end{equation*}
In the following we abbreviate $\mathbf{X}_{G,P} := \Ind_P^G(\mathbf{X}_M)^{I_G}$, and we always use the functor $\Ind_{H_M}^{H_G}$ in the (naturally isomorphic) form
\begin{equation*}
  Y \longmapsto \mathbf{X}_{G,P} \otimes_{H_M} Y \ .
\end{equation*}
It is a standard fact (cf.\ \cite{CE} Prop.\ II.5.2) that this latter functor has the right adjoint functor
\begin{align*}
  \Mod(H_G) & \longrightarrow \Mod(H_M) \\
          Z & \longmapsto \Hom_{H_G}(\mathbf{X}_{G,P},Z) \ .
\end{align*}
We also quote from \cite{OV} Thm.\ 4.1 and Cor.\ 4.6 the following results:
\begin{itemize}
  \item[(F)] $\Ind_{H_M}^{H_G}$ is faithful and exact, or equivalently, $\mathbf{X}_{G,P}$ as an $H_M$-module is faithfully flat.
  \item[(G)] $\Ind_{H_M}^{H_G}$ has an exact left adjoint and, as a consequence, preserves injective objects.
  \item[(H)] The diagram
\begin{equation}\label{diag:Ind-tensor}
  \xymatrix{
    \Mod(H_M) \ar[d]_{\mathbf{X}_{G,P} \otimes_{H_M} -} \ar[rr]^{\mathbf{X}_M \otimes_{H_M} -} && \Mod_k(M) \ar[d]^{\Ind_P^G} \\
    \Mod(H_G) \ar[rr]^{\mathbf{X}_G \otimes_{H_G} -} && \Mod_k(G)   }
\end{equation}
is commutative (up to natural isomorphism).
\end{itemize}

To pass to the derived picture we fix injective resolutions $\mathbf{X}_G \xrightarrow{\simeq} \mathcal{I}_G^\bullet$ and $\mathbf{X}_M \xrightarrow{\simeq} \mathcal{I}_M^\bullet$ in $\Mod_k(G)$ and $\Mod_k(M)$, respectively. We consider the differential graded algebras
\begin{equation*}
  H_G^\bullet := \End_{\Mod_k(G)}^\bullet(\mathcal{I}_G^\bullet)^{\mathrm{op}}    \qquad\text{and}\qquad    H_M^\bullet := \End_{\Mod_k(M)}^\bullet(\mathcal{I}_M^\bullet)^{\mathrm{op}}
\end{equation*}
and their derived categories $D(H_G^\bullet)$ and $D(H_M^\bullet)$, respectively (compare \cite{Sch} \S3). A first guess, based upon the definition of $\mathbf{X}_{G,P}$, would be to consider the $(H_G^\bullet,H_M^\bullet)$-bimodule
\begin{equation*}
   \mathcal{I}_{G,P}^\bullet := \Hom_{\Mod_k(G)}^\bullet (\mathcal{I}_G^\bullet, \Ind_P^G( \mathcal{I}_M^\bullet)) \ .
\end{equation*}
It gives rise to the pair of exact functors between triangulated categories
\begin{align*}
  \mathcal{I}_{G,P}^\bullet \otimes_{H_M^\bullet}^{\mathbf{L}} - : D(H_M^\bullet) & \longrightarrow D(H_G^\bullet) \\
                                                     Y^\bullet & \longmapsto \mathbf{q} (\mathcal{I}_{G,P}^\bullet \otimes_{H_M^\bullet} \mathbf{p} Y^\bullet)
\end{align*}
and
\begin{align*}
  \mathbf{R}\Hom_{H_G^\bullet} (\mathcal{I}_{G,P}^\bullet,-) : D(H_G^\bullet) & \longrightarrow D(H_M^\bullet) \\
  Z^\bullet & \longmapsto \mathbf{q} (\Hom_{H_G^\bullet}^\bullet (\mathcal{I}_{G,P}^\bullet,\mathbf{i} Z^\bullet)) \ ,
\end{align*}
the first one being left adjoint to the second one (cf.\ \cite{Kel} \S2.6). Here $K(H_M^\bullet)$ denotes the homotopy category of $H_M^\bullet$-modules and $\mathbf{p} : D(H_M^\bullet) \longrightarrow K(H_M^\bullet)$, resp.\ $\mathbf{i} : D(H_M^\bullet) \longrightarrow K(H_M^\bullet)$, is a fully faithful left, resp.\ right, adjoint of the quotient functor $\mathbf{q} : K(H_M^\bullet) \longrightarrow D(H_M^\bullet)$ such that the adjunction morphism $\mathbf{p} \circ \mathbf{q} (?) \xrightarrow{\simeq} \id(?)$, resp.\ $\id(?) \xrightarrow{\simeq} \mathbf{i} \circ \mathbf{q} (?)$, is a h-projective, resp.\ h-injective,  resolution of $?$ (cf.\ \cite{Kel} \S1.2 and \S2.5).

At this point we impose the \textbf{assumption} that the pro-$p$ group $I_G$ is $p$-torsionfree. (This forces $\mathfrak{F}$ to have characteristic zero.) Of course, $I_M$ then is $p$-torsionfree as well. By \cite{Sch} Thm.\ 3 we then have the equivalence of triangulated categories
\begin{align}\label{f:equiv}
  h_G : D(G) & \xrightarrow{\;\simeq\;} D(H_G^\bullet) \\
   V^\bullet & \longmapsto \mathbf{q} (\Hom_{\Mod_k(G)}^\bullet (\mathcal{I}_G^\bullet, \mathbf{i} V^\bullet)) \ ,   \nonumber
\end{align}
for which a quasi-inverse functor is given by
\begin{align*}
  t_G : D(H_G^\bullet) & \xrightarrow{\;\simeq\;} D(G) \\
  Z^\bullet & \longmapsto \mathbf{q} (\mathcal{I}_G^\bullet \otimes_{H_G^\bullet} \mathbf{p} Z^\bullet) \ .
\end{align*}
In the same way we have the quasi-inverse equivalences $h_M$ and $t_M$. Here, similarly as before, $K(G)$ denotes the homotopy category of unbounded complexes in $\Mod_k(G)$ and $\mathbf{i} : D(G) \longrightarrow K(G)$ is a fully faithful right adjoint of the quotient functor $\mathbf{q} : K(G) \longrightarrow D(G)$ such that the adjunction morphism $\id(?) \xrightarrow{\simeq} \mathbf{i} \circ \mathbf{q} (?)$ is a h-injective resolution of $?$ (\cite{Sch} \S2).

Our goal is to show that under these equivalences the functor $\Ind_P^G$ corresponds to the functor $\mathcal{I}_{G,P}^\bullet \otimes_{H_M^\bullet}^{\mathbf{L}} -$. But one quickly realizes that this can not be true with the above naive definition of $\mathcal{I}_{G,P}^\bullet$. One needs to replace $\Ind_P^G( \mathcal{I}_M^\bullet)$ by an injective resolution and change the definition to $\mathcal{I}_{G,P}^\bullet := \Hom_{\Mod_k(G)}^\bullet (\mathcal{I}_G^\bullet, \mathbf{i}\Ind_P^G( \mathcal{I}_M^\bullet))$. This runs into the problem, though, that $H_M^\bullet$ might no longer act on $\mathbf{i}\Ind_P^G( \mathcal{I}_M^\bullet)$. To make a good choice of the functor $\mathbf{i}$ we have to lift the setting to dg-categories.

\section{The enhanced setting}\label{sec:enhanced-set}

\subsection{The enhancement of $D(G)$}\label{subsec:enhanc-DG}

We let $C(G)$ denote the Grothendieck abelian category of unbounded cochain complexes in $\Mod_k(G)$. This category can be enriched to the dg-category $\underline{C}(G)$ with the same objects as $C(G)$ but with morphisms the cochain complexes of $k$-vector spaces $\Hom_{\Mod_k(G)}^\bullet(V_1^\bullet,V_2^\bullet)$. The category $K(G)$ introduced earlier, is the homotopy category of $\underline{C}(G)$.

According to \cite{HA} Prop.\ 1.3.5.3 the category $C(G)$ has a left proper combinatorial (closed) model structure for which
\begin{itemize}
  \item[--] the cofibrations are the monomorphisms,
  \item[--] the weak equivalences are the quasi-isomorphisms, and
  \item[--] the fibrations are the morphisms which have the right lifting property with respect to all trivial cofibrations.
\end{itemize}
It is called the injective model structure of $C(G)$.

\begin{proposition}\label{inj-model}
  A morphism in $C(G)$ is a fibration if and only if it is an epimorphism with an h-injective and degreewise injective kernel.
\end{proposition}
\begin{proof}
See \cite{Gil} Cor.\ 7.1.
\end{proof}

Let $\underline{I}(G)$ denote the full dg-subcategory of $\underline{C}(G)$ whose objects are fibrant. The crucial technical fact which we will be using is the following special case of a theorem of Schn\"urer.

\begin{theorem}\label{enriched-replac}
  There exists a dg-functor $\underline{i} : \underline{C}(G) \longrightarrow \underline{I}(G)$ (called the fibrant replacement functor) and a natural transformation $\phi : \id \rightarrow \mathrm{inclusion} \circ \underline{i}$ such that $\phi_{V^\bullet} : V^\bullet \rightarrow \underline{i}(V^\bullet)$, for any $V^\bullet$ in $C(G)$, is a trivial cofibration.
\end{theorem}
\begin{proof}
We recall from \cite{Sch} Lemma 1 that $\Mod_k(G)$ is a complete Grothendieck abelian category. In particular, the abelian category $C(G)$ is complete and cocomplete. On the other hand, the category of  cochain complexes of $k$-vector spaces can be viewed as the full subcategory of those complexes in $C(G)$ whose terms carry the trivial $G$-action. Using this observation  one easily shows that $\underline{C}(G)$ is tensored and cotensored in the sense of \cite{Schn} 3.2.5. Hence $\Mod_k(G)$ satisfies the assumptions of \cite{Schn} Thm. 4.3.
\end{proof}

\subsection{The commutative diagrams}\label{subsec:derived-comm}

We choose a fibrant replacement functor $\underline{i}$ as in Thm.\ \ref{enriched-replac} and finally introduce the correct $(H_G^\bullet,H_M^\bullet)$-bimodule
\begin{equation*}
   \mathcal{I}_{G,P}^\bullet := \Hom_{\Mod_k(G)}^\bullet (\mathcal{I}_G^\bullet, \underline{i} \Ind_P^G( \mathcal{I}_M^\bullet)) \ .
\end{equation*}

\begin{proposition}\label{derived-comm}
The diagrams
\begin{equation*}
  \xymatrix{
    D(M) \ar[d]_{\Ind_P^G}  \ar[rr]^{h_M} && D(H_M^\bullet) \ar[d]^{\mathcal{I}_{G,P}^\bullet \otimes_{H_M^\bullet}^{\mathbf{L}} -} \\
    D(G) \ar[rr]^{h_G} &&   D(H_G^\bullet)  }
\end{equation*}
and
\begin{equation*}
  \xymatrix{
    D(H_M^\bullet) \ar[d]_{\mathcal{I}_{G,P}^\bullet \otimes_{H_M^\bullet}^{\mathbf{L}} -} \ar[rr]^{t_M} && D(M) \ar[d]^{\Ind_P^G} \\
    D(H_G^\bullet) \ar[rr]^{t_G} && D(G)   }
\end{equation*}
are commutative (up to natural isomorphism).
\end{proposition}
\begin{proof}
The commutativity of the second diagram is a consequence of the commutativity of the first one (and vice versa). In order to treat the first diagram we will make use of the following general principle.

Let $\mathcal{S}$ and $\mathcal{T}$ be two triangulated categories with arbitrary direct sums and such that $\mathcal{S}$ has a compact generator $C$. Let $F_i : \mathcal{S} \longrightarrow \mathcal{T}$, for $i = 0, 1$, be two exact functors which respect arbitrary direct sums. Finally let $\tau : F_1 \longrightarrow F_2$ be a natural transformation. Then $\tau$ is a natural isomorphism if and only if $\tau_C : F_1(C) \longrightarrow F_2(C)$ is an isomorphism.

We briefly recall the argument (\cite{Kel} \S1.4 last Remark b)). Let $\mathcal{S}_0$ be the full subcategory of those objects $Y$ in $\mathcal{S}$ for which $\tau_Y$ is an isomorphism. Then $\mathcal{S}_0$ is a strictly full triangulated subcategory of $\mathcal{S}$ which contains the compact generator $C$ and is closed under the formation of arbitrary direct sums. Hence, by the principle of infinite d\'evissage, it must be equal to $\mathcal{S}$.

This is applicable in our situation. First of all we note that the quotient functor $\mathbf{q}$ in both cases commutes with arbitrary direct sums (cf.\ \cite{KS} \S14.3. first paragraph). The derived category $D(M)$ has arbitrary direct sums and the compact generator $\mathbf{X}_M \xrightarrow{\simeq} \mathcal{I}_M^\bullet$ by \cite{Sch} Remark 2, Lemma 4, and Prop.\ 6. The derived category $D(H_G^\bullet)$ has arbitrary direct sums (cf.\ \cite{Kel} \S2.5). The functor $\Ind_M^G$ commutes with arbitrary direct sums as a consequence of property (C). Since arbitrary direct sums of $K$-projective objects are $K$-projective and since the tensor project commutes with arbitrary direct sums, the same holds for the derived tensor product $\mathcal{I}_{G,P}^\bullet \otimes_{H_M^\bullet}^{\mathbf{L}} -$. The functors $h_M$ and $h_G$ commute with arbitrary direct sums as a consequence of \cite{Sch} Lemma 3.

The functors in question are
\begin{align*}
  F_1(-) & := \mathbf{q} (\mathcal{I}_{G,P}^\bullet \otimes_{H_M^\bullet} \mathbf{p} \mathbf{q} \Hom_{\Mod_k(M)}^\bullet(\mathcal{I}_M^\bullet, \underline{i} -)) \\
         & = \mathbf{q} (\Hom_{\Mod_k(G)}^\bullet (\mathcal{I}_G^\bullet, \underline{i} \Ind_P^G( \mathcal{I}_M^\bullet)) \otimes_{H_M^\bullet} \mathbf{p} \mathbf{q} \Hom_{\Mod_k(M)}^\bullet(\mathcal{I}_M^\bullet, \underline{i} -))
\end{align*}
and
\begin{equation*}
  F_2(-) := \mathbf{q} \Hom_{\Mod_k(G)}^\bullet (\mathcal{I}_G^\bullet, \underline{i} \Ind_P^G(-)) \ .
\end{equation*}
The natural transformation $\tau : F_1 \longrightarrow F_2$ is obtained by combining the following natural transformations:
\begin{itemize}
  \item $F_1(-) \longrightarrow  \mathbf{q} (\mathcal{I}_{G,P}^\bullet \otimes_{H_M^\bullet} \Hom_{\Mod_k(M)}^\bullet(\mathcal{I}_M^\bullet, \underline{i} -))$ induced by the adjunction $\mathbf{p} \mathbf{q} \longrightarrow \id$;
  \item the image under $\mathbf{q}$ of the natural pairing
\begin{align*}
  \Hom_{\Mod_k(G)}^\bullet (\mathcal{I}_G^\bullet, \underline{i}\Ind_P^G( \mathcal{I}_M^\bullet)) \otimes_{H_M^\bullet} \Hom_{\Mod_k(M)}^\bullet(\mathcal{I}_M^\bullet, \underline{i} -) & \longrightarrow \Hom_{\Mod_k(G)}^\bullet (\mathcal{I}_G^\bullet, \underline{i} \Ind_P^G(\underline{i} -)) \\
   \alpha \otimes \beta & \longmapsto \underline{i}(\Ind_P^G(\beta) \circ \alpha) \ ;
\end{align*}
  \item the inverse of the image under $\mathbf{q}$ of the homotopy equivalence
\begin{equation*}
  \Hom_{\Mod_k(G)}^\bullet (\mathcal{I}_G^\bullet, \underline{i} \Ind_P^G( -)) \longrightarrow \Hom_{\Mod_k(G)}^\bullet (\mathcal{I}_G^\bullet, \underline{i} \Ind_P^G(\underline{i} -))  \ :
\end{equation*}
Note that by Thm. \ref{enriched-replac}, we have a natural transformation $\id \to \underline{i}$. Applying $\underline{i} \Ind_P^G$ yields the natural transformation
\begin{equation*}
  \underline{i} \Ind_P^G (-) \longrightarrow \underline{i} \Ind_P^G (\underline{i} -) \ .
\end{equation*}
Evaluated on objects, it becomes a quasi-isomorphism between fibrant objects and hence a homotopy equivalence. Applying $\mathbf{q} \circ \Hom_{\Mod_k(G)}^\bullet (\mathcal{I}_G^\bullet, -)$ therefore is a natural isomorphism.
\end{itemize}
On the compact generator $\mathbf{X}_M$ this simplifies to the inverse of the image under $\mathbf{q}$ of the homotopy equivalence
\begin{equation*}
   \Hom_{\Mod_k(G)}^\bullet (\mathcal{I}_G^\bullet, \underline{i} \Ind_P^G( \mathbf{X}_M^\bullet)) \longrightarrow  \Hom_{\Mod_k(G)}^\bullet (\mathcal{I}_G^\bullet, \underline{i} \Ind_P^G( \mathcal{I}_M^\bullet)) = \mathcal{I}_{G,P}^\bullet
\end{equation*}
induced by the trivial cofibration $\mathbf{X}_M \rightarrow \underline{i} \mathbf{X}_M = \mathcal{I}_M^\bullet$.
\end{proof}

\section{Adjoint functors}

\subsection{The right adjoints}\label{subsec:right-adjoints}

Since $\Ind_P^G : \Mod_k(M) \longrightarrow \Mod_k(G)$ commutes with arbitrary direct sums it has a right adjoint functor $R_P^G : \Mod_k(G) \longrightarrow \Mod_k(M)$. Moreover, being a right adjoint the functor $R_P^G$ necessarily is left exact. In particular, its right derived functors exist. In fact, by \cite{KS} Thm.\ 14.3.1, the functor $\Ind_P^G : D(M) \longrightarrow D(G)$ has a right adjoint functor $\mathbf{R}R_P^G : D(G) \longrightarrow D(M)$. It is, of course, the total derived functor of $R_P^G$.

On the other hand, we have recalled already (cf.\ \cite{Kel} \S2.6) that the functor  $\mathcal{I}_{G,P}^\bullet \otimes_{H_M^\bullet}^{\mathbf{L}} - : D(H_M^\bullet)  \longrightarrow D(H_G^\bullet)$ has the right adjoint functor $\mathbf{R}\Hom_{H_G^\bullet} (\mathcal{I}_{G,P}^\bullet,-) : D(H_G^\bullet) \longrightarrow D(H_M^\bullet)$.

\begin{corollary}\label{right-adj-comm}
The diagrams
\begin{equation*}
  \xymatrix{
    D(G) \ar[d]_{\mathbf{R}R_P^G}  \ar[rr]^{h_G} && D(H_G^\bullet) \ar[d]^{\mathbf{R}\Hom_{H_G^\bullet} (\mathcal{I}_{G,P}^\bullet,-)} \\
    D(M) \ar[rr]^{h_M} &&   D(H_M^\bullet)  }
\end{equation*}
and
\begin{equation*}
  \xymatrix{
    D(H_G^\bullet) \ar[d]_{\mathbf{R}\Hom_{H_G^\bullet} (\mathcal{I}_{G,P}^\bullet,-)} \ar[rr]^{t_G} && D(G) \ar[d]^{\mathbf{R}R_P^G} \\
    D(H_M^\bullet) \ar[rr]^{t_M} && D(M)   }
\end{equation*}
are commutative (up to natural isomorphism).
\end{corollary}
\begin{proof}
This is a purely formal consequence of Prop.\ \ref{derived-comm} and the equivalences \eqref{f:equiv}.
\end{proof}

It is an open problem to compute the cohomology $\mathbf{R}^i R_P^G(V)$ of the right adjoint even for a single representation $V$. If $V$ is admissible this might be related to the construction of the higher ordinary parts functors in \cite{Eme}.

\subsection{The left adjoints}\label{subsec:left-adjoints}

In \S1 (G) we have noted that the functor $\mathbf{X}_{G,P} \otimes_{H_M} - : \Mod(H_M) \longrightarrow \Mod(H_G)$ has a left adjoint functor, which is exact. In order to compute this left adjoint we temporarily denote it by $\ell(-)$. Since $H_G$ is a $(H_G,H_G)$-bimodule we have that $\ell(H_G)$ is a $(H_M,H_G)$-bimodule. We consider the natural isomorphism of adjunction
\begin{equation*}
  \Hom_{H_G}(Z,\mathbf{X}_{G,P} \otimes_{H_M} Y) \cong \Hom_{H_M}(\ell(Z),Y) \ .
\end{equation*}
Taking $Z$ to be $H_G$ we deduce a natural isomorphism
\begin{equation*}
  \mathbf{X}_{G,P} \otimes_{H_M} Y \cong \Hom_{H_M}(\ell(H_G),Y) \ .
\end{equation*}
Recall from \S1 (F) that $\mathbf{X}_{G,P}$ is a flat right $H_M$-module. Hence $\ell(H_G)$ must be a projective left $H_M$-module. It also follows that the functor $Y \longmapsto \Hom_{H_M}(\ell(H_G),Y)$ commutes with arbitrary direct sums. This implies\footnote{Here is the argument: Let $Q$ be a projective left module over the ring $A$ such that the functor $\Hom_A(Q,-)$ commutes with arbitrary direct sums. Write $Q \oplus Q' = \oplus_{j \in J} A$ as a direct summand of a free module. Let $\alpha : Q \hookrightarrow \oplus_{j \in J} A$ denote the inclusion map. Then $\alpha \in \Hom_A(Q,\oplus_{j \in J} A) = \oplus_{j \in J} \Hom_A(Q,A)$. Hence $\alpha$ is the sum of finitely many $\alpha_j : Q \rightarrow A$, which means that $Q \subseteq \oplus_{j \in J_0} A$ for some finite subset $J_0 \subseteq J$. It follows that $Q \oplus (Q' \cap \oplus_{j \in J_0} A) = \oplus_{j \in J_0} A$ and therefore that $Q$ is finitely generated.} that $\ell(H_G)$ is finitely generated as an $H_M$-module. Finally, by taking $Y$ to be $H_M$ in the last natural isomorphism, we obtain an isomorphism $\mathbf{X}_{G,P} \cong \Hom_{H_M}(\ell(H_G),H_M)$ and then also an isomorphism $\ell(H_G) \cong \Hom_{H_M}(\mathbf{X}_{G,P},H_M)$. This proves the following fact.

\begin{lemma}\label{underived-left-adj}
  $\mathbf{X}_{G,P}$, as a right $H_M$-module, is finitely generated projective, and the functor $\mathbf{X}_{G,P} \otimes_{H_M} - $ has the left adjoint functor $\Hom_{H_M}(\mathbf{X}_{G,P},H_M) \otimes_{H_G} -$.
\end{lemma}

It needs to be pointed out, though, that the diagram
\begin{equation*}
  \xymatrix{
    \Mod_k(G) \ar[d]_{(-)_N}  \ar[rr]^{(-)^{I_G}} && \Mod(H_G) \ar[d]^{\Hom_{H_M}(\mathbf{X}_{G,P},H_M) \otimes_{H_G} -} \\
    \Mod_k(M) \ar[rr]^{(-)^{I_M}} &&   \Mod(H_M)  }
\end{equation*}
is NOT commutative as shown in \cite{OV} Cor.\ 4.11.

At this point we recall that the category $\Mod_k(G)$ has arbitrary direct products which are constructed in the following way. Let $(V_i)_{i \in I}$ be a family of smooth $G$-representations. Its direct product in $\Mod_k(G)$
\begin{align*}
  \prod^G_{i \in I} V_i  & := \text{smooth part of $\prod_{i \in I} V_i$}  \\
   & = \bigcup_{U \subseteq G} (\prod_{i \in I} V_i)^U  \qquad\text{(with $U$ running over all open subgroups of $G$)}
\end{align*}
is the smooth part of the cartesian product of the $V_i$ with the diagonal $G$-action. It is a formal triviality that direct products of injective objects are injective.

\begin{lemma}\label{Ind-product}
  The functor $\Ind_P^G : \Mod_k(M) \longrightarrow \Mod_k(G)$ commutes with arbitrary direct products.
\end{lemma}
\begin{proof}
Let $(E_i)_{i \in I}$ be a family of smooth $M$-representations. The projection maps $\prod^M_i E_i \rightarrow E_j$ induce to homomorphisms $\Ind_P^G(\prod^M_i E_i) \rightarrow \Ind_P^G(E_j)$ in $\Mod_k(G)$, which, by the universal property of the direct product, combine into a homomorphism $\Ind_P^G(\prod^M_i E_i) \rightarrow \prod^G_i \Ind_P^G(E_i)$. It obviously is injective. To see its surjectivity let $(f_i)_i \in (\prod_i \Ind_P^G(E_i))^U$, for some open subgroup $U \subseteq G$, be an element. Each $f_i$ is a left $U$-invariant function $f_i : G \rightarrow E_i$ such that $f_i(gh) = h^{-1}(f_i(g))$ for any $g \in G$ and $h \in P$. We now define the function $f : G \rightarrow \prod_i E_i$ by $f(g) := (f_i(g))_i$. It is left $U$-invariant and therefore locally constant on $G$. It satisfies $f(gh) = h^{-1}(f(g))$ for any $g \in G$ and $h \in P$. In order to see that $f$ is a preimage under the map in question of $(f_i)_i$ it remains to establish that the values of $f$ lie in $\prod^M_i E_i$. Fix a $g \in G$. Then $U_g := g^{-1} U g \cap M$ is an open subgroup of $M$. For $h = g^{-1} u g \in U_g$ we have $f(g) = f(ug) = f(ghg^{-1} g) = f(gh) = h^{-1}(f(g)$. This shows that the value $f(g)$ is $U_g$-invariant.
\end{proof}

We note that forming an infinite direct product in $\Mod_k(G)$ is not an exact operation because the passage to $U$-invariants is not exact. But we do have the following.

\begin{lemma}\label{DG-product}
 The category $D(G)$ has arbitrary direct products.
\end{lemma}
\begin{proof}
Apply \cite{KS} Thm.\ 14.2.1 and Cor.\ 10.5.3.
\end{proof}

It is obvious that the category $C(G)$ has direct products, and they are formed in the naive way:
\begin{equation*}
  \prod^G_{i \in I} V^\bullet_i : \quad  \ldots \longrightarrow \prod^G_{i \in I} V^n_i \longrightarrow \prod^G_{i \in I} V^{n+1}_i \longrightarrow  \ldots
\end{equation*}
The homotopy category $K(G)$ then has direct products as well, which can be computed in $C(G)$. But, due to the failure of direct products in $\Mod_k(G)$ being exact, an infinite direct product of acyclic complexes need not to be acyclic. Therefore the natural functor $K(G) \longrightarrow D(G)$ need not preserve infinite direct products. How do we compute then direct products in $D(G)$? For this we go back to the injective model structure on $C(G)$ introduced in section \ref{subsec:enhanc-DG} and let $I(G)$ denote the full subcategory of $C(G)$ consisting of the fibrant objects and $h(\underline{I}(G))$ the homotopy category of the dg-category $\underline{I}(G)$. On the one hand $h(\underline{I}(G))$ is the category with the same objects as $I(G)$ but with morphisms being the homotopy classes of morphisms in $I(G)$. On the other hand the natural functor $h(\underline{I}(G)) \xrightarrow{\simeq} D(G)$ is an equivalence of categories.

\begin{lemma}\label{inj-closed}
  The category $I(G)$ is closed under arbitrary direct products in $C(G)$.
\end{lemma}
\begin{proof}
Because of Prop.\ \ref{inj-model} we need to check that direct products of h-injective complexes are h-injective. But this is straightforward from the definition.
\end{proof}

This result implies that, in order to form direct products in $D(G)$, one has to pass to fibrant resolutions and then take their naive direct product.

We warn the reader that neither the natural functor $\Mod_k(G)\longrightarrow D(G)$ nor the cohomology functor $h^* : D(G) \longrightarrow \Mod_k(G)$ commutes with arbitrary direct products.

But we have the following remarkable result.

\begin{proposition}[C.\ Heyer]
   The functor $Ind^G_P : D(M) \longrightarrow D(G)$ commutes with arbitrary direct products.
\end{proposition}

By an application of Brown representability this has the following consequence.

\begin{corollary}
   The functor $Ind^G_P : D(M) \longrightarrow D(G)$ has a left adjoint.
\end{corollary}


\begin{thebibliography}{CFKS}

\bibitem[CE]{CE}
Cartan H., Eilenberg S.: \emph{Homological Algebra}. Princeton Univ.\ Press 1956

\bibitem[Eme]{Eme}
Emerton M.: \emph{Ordinary par ts of admissible representations of $p$-adic reductive groups II. Derived functors}. Ast\'erisque 331, 403-459 (2010)

\bibitem[Gil]{Gil}
Gillespie J.: \emph{Kaplansky classes and derived categories}. Math.\ Z.\ 257, 811-843 (2007)

\bibitem[Hey]{Hey}
Heyer C.: \emph{Produkte in $D(G)$ und parabolische Induktion}. Notes 2020

\bibitem[KS]{KS}
Kashiwara M., Shapira P.: \emph{Categories and sheaves}. Springer 2006

\bibitem[Kel]{Kel}
Keller B.: \emph{On the construction of triangle equivalences}. In ``Derived equivalences for Group Rings'' (Eds.\ K\"onig, Zimmermann), Springer Lect.\ Notes Math., vol.\ 1685, pp.\ 155-176 (1998)

\bibitem[HA]{HA}
Lurie J.: \emph{Higher Algebra}.

\bibitem[OV]{OV}
Ollivier R., Vigneras M.-F.: \emph{Parabolic induction in characteristic $p$}. Selecta Math.\ New Ser.\ 24(5), 3973-4039 (2018)

\bibitem[Sch]{Sch}
Schneider P.: \emph{Smooth representations and Hecke modules in characteristic $p$}. Pacific J.\ Math.\ 279, 447-464 (2015)

\bibitem[Schn]{Schn}
Schn\"urer O.: \emph{Six operations on dg enhancements of derived categories of sheaves}. Selecta Math.\ New Ser.\ 24(3), 1805-1911 (2018)

\bibitem[Vig]{Vig}
Vigneras M.-F.: \emph{The right adjoint of the parabolic induction}. Progress in Math.\ 319, 405-424, Birkh\"auser 1996

\end{thebibliography}
\end{document}